\documentclass[12pt,a4paper]{amsart}
 \pdfoutput=1
 
 \usepackage{amsmath,amssymb,amsthm}
 \usepackage{bm}  
 \usepackage{mathtools}
 \usepackage{thmtools}
 \usepackage{float} 
 
 \usepackage{xcolor} 	
 \usepackage{hyperref}
 \hypersetup{
 	colorlinks,
     linkcolor={red!60!black},
     citecolor={green!60!black},
     urlcolor={blue!60!black},
 }
 \usepackage[maxbibnames=99]{biblatex}
 \usepackage[utf8]{inputenc}
  \usepackage{csquotes}
 \usepackage[T1]{fontenc}
 \usepackage{lmodern}
 \usepackage[babel]{microtype}
 \usepackage[english]{babel}
 
 \usepackage{geometry}
 \usepackage{enumitem}

\theoremstyle{plain}
\newtheorem{theorem}{Theorem}[section]

\newtheorem{corollary}[theorem]{Corollary}

\newtheorem{observation}[theorem]{Observation}

\newtheorem*{theorem*}{Theorem}
\newtheorem*{corollary*}{Corollary}

\theoremstyle{definition}

\title{A Note on Matching Variables to Equations}

\author{Attila Jo\'{o}}
\thanks{Funded by the Deutsche Forschungsgemeinschaft (DFG, German Research Foundation) - Grant No. 513023562 and partially by NKFIH 
OTKA-129211}
\address{Attila Jo\'{o},
Department of Mathematics, University of Hamburg, Bundesstra{\ss}e 55 (Geomatikum), 20146 Hamburg, Germany}
\email{attila.joo@uni-hamburg.de}

\begin{document}

\begin{abstract}
We showed with J. P. Gollin that if a (possibly infinite) homogeneous linear equation system has only the trivial solution, then there exists 
an injective function from the variables to the equations such that each variable appears with  non-zero coefficient in its image. Shortly after, a 
more 
elementary proof was found by Aharoni and Guo. In this note we present a very short matroid-theoretic proof of this theorem.
\end{abstract}
\maketitle

\section{Introduction}
Let $ \mathbb{F} $ be a field, and $ A \in \mathbb{F}^{m \times n} $. Suppose that the homogenous  linear 
equation system $ Ax=0 $ has only 
the trivial solution $ x=0 $. Then for each subset $ C $ of the columns of $ A $, there must be at least $ \left|C\right| $ rows where at least one 
column in $ C $ have a non-zero element. Indeed, otherwise, the columns in $ C $ are linearly dependent which leads to a non-trivial solution of $ 
Ax=0 $. It follows by Hall's theorem that the columns can be matched to the rows along non-zero elements. More precisely, there is an injection $ 
\varphi: [n]\rightarrow [m] $ such that $ a_{\varphi(j),j}\neq 0 $ for every $ j \in [n] $. 

We investigated with J. P. Gollin if this remains true for infinite equation systems (with finitely many variables in each equation). 
Although Hall's  marriage theorem has a certain generalization for infinite graphs (see \cite[Theorem 3.2]{aharoni1991infinite}), the argument 
above does not seem to be adaptable to the infinite case. Even 
so, other tools in infinite matching theory (namely \cite[Theorem 1]{wojciechowski1997criteria})  let us answer the question affirmatively:
\begin{theorem}\label{thm: main}
 Let $ \mathbb{F} $ be a field, and let $ a: I \times J \rightarrow \mathbb{F} $ be a function such that 
for each fixed $ i\in I $, there are only finitely many $ j \in J $ such that $ a_{i,j}\neq 0 $. Suppose that the homogenous linear equation system 
defined by $ a $ has only the trivial solution, i.e.  $ \sum_{j\in J}a_{i,j}x_j=0 $ for each $ i\in I $ for a function $ x: J\rightarrow \mathbb{F} $
only if $ x $ is constant zero. Then there is an injection $\varphi: J\rightarrow I $ such that $ a_{\varphi(j),j}\neq 0 $ for every $ j\in J $.
\end{theorem}

 Shortly 
after, a more 
elementary proof was found by Aharoni 
and Guo \cite{aharoni2024tight}. They also pointed out a short matroid-theoretic proof in the finite case. This made us wonder if a similar 
matroid-theoretic approach
could be successful for 
infinite equation systems.  The key tools we use in this note are the so called thin-sum matroids introduced by Bruhn and Diestel 
(see \cite[Theorem 18]{bruhn2011infmatrgraph}) and further investigated by Afzali and  Bowler in \cite{borujeni2015thin}. These together with 
(the dualization of) a base 
exchange 
property of infinite matroids due to Aharoni and Pouzet  \cite[Theorem 2.1]{aharoni1991bases} leads to a short proof of Theorem 
\ref{thm: main}. The aim of this note is to present this proof.

\section{Preliminaries}
Infinite matroids were introduced by Higgs in the late 1960s \cite{higgs1969matroids}.  Several decades later the same concept was discovered 
independently by Bruhn et al. 
in 
\cite{bruhn2013axioms} together with the following axiomatization:
A \emph{matroid} is an ordered pair  $M=(E,\mathcal{I})$ with ${\mathcal{I} \subseteq 
\mathcal{P}(E)}$ such that  
\begin{enumerate}
[label=(\Roman*)]
\item $ \emptyset\in  \mathcal{I} $;
\item $ \mathcal{I} $ is closed under taking subsets;
\item\label{item: matroid axiom3} For every $ I,J\in \mathcal{I} $ where  $J $ is $ \subseteq $-maximal in $ \mathcal{I} $ and $ I 
$ is 
not, there exists an $  e\in J\setminus I $ such that
$ I\cup \{ e \}\in \mathcal{I} $;
\item\label{item: matroid axiom4} For every $ X\subseteq E $, any $ I\in \mathcal{I}\cap 
\mathcal{P}(X)  $ can be extended to a $ \subseteq $-maximal element of 
$ \mathcal{I}\cap \mathcal{P}(X) $.
\end{enumerate}
The sets 
in~$\mathcal{I}$ are called 
\emph{independent} while the sets in ${\mathcal{P}(E) \setminus \mathcal{I}}$ are 
\emph{dependent}. The maximal independent sets (exists by \ref{item: matroid axiom4}) are called \emph{bases}. 
The 
minimal 
dependent sets are the \emph{circuits}. Every dependent set contains a circuit (which is a non-trivial fact for infinite 
matroids). For an  ${X \subseteq E}$, the 
pair ${M  
\upharpoonright X :=(X,\mathcal{I} 
\cap \mathcal{P}(X))}$ is a matroid and it is called the \emph{restriction} of~$M$ to~$X$. We write $M - X$ for 
$ M  \upharpoonright (E\setminus X) $  and call it the minor obtained by the 
\emph{deletion} of~$X$. 
The \emph{contraction} of $ X $ in $ M $ is the matroid $ M/X $ on $ E\setminus X $ in which $ I\subseteq E\setminus X $ is independent iff $ 
J\cup I $ is 
independent in $ M $ for a (equivalently: for every) set  $ J $ that is maximal among the independent subset of $ X $.
Contraction and deletion commute, i.e. for 
disjoint 
$ X,Y\subseteq E $, we have $ (M/X)-Y=(M-Y)/X $.  Matroids of this form are the  \emph{minors} of~$M$. We say that~${X 
\subseteq E}$ \emph{spans}~${e \in E}$ in matroid~$M$ if either~${e \in X}$ or $ \{ e \} $ is dependent in $ M/X $. The dual $ M^{*} $ of $ M 
$ is the matroid on the same ground set in which a set is independent if it is disjoint from a base of $ M $.  Contraction and deletion are related by 
duality in the following way: $ (M/X-Y)^{*}=M^{*}/Y-X $.
A matroid is called \emph{finitary} if all its circuits are finite and it is \emph{cofinitary} if its dual is finitary. The class of finitary matroids is 
closed under taking minors and thus so is the class of cofinitary matroids.

\begin{theorem}[{Aharoni and Pouzet, \cite[Theorem 2.1]{aharoni1991bases}}]\label{thm: AP Base exch}
If $ M $ is a finitary matroid, then for every bases $ B_0 $ and $ B_1 $ of $ M $, there is a bijection $ f: B_0 
\rightarrow B_1 $ 
such that $ B_0\setminus \{ x \} \cup \{ f(x) \} $ is a base of $ M $ for each $ x\in B_0 $.
\end{theorem}
\noindent For short proofs of generalizations of Theorem \ref{thm: AP Base exch} using the same notation we are using in this article, see 
\cite{janko2023baseexch}. 

The following concept of thin-sum 
matroids  was 
introduced by Bruhn and Diestel 
(see \cite[Theorem 18]{bruhn2011infmatrgraph}): Let $ X $ be a set and let $ \mathbb{F}$ be a field. A family $ \{ f_e:\ e\in E \} $ of $ 
X\rightarrow \mathbb{F} $ functions is called 
\emph{thin} if for each $ x\in X $ there are only finitely many $ e\in E $ with $ f_e(x)\neq 0 $. Note that for any $ \lambda: E \rightarrow 
\mathbb{F} $, the function $ \sum_{e\in E}\lambda_e f_e $ is well-defined pointwise. An $ I\subseteq E $ is \emph{thin independent} when 
$ \sum_{e\in I}\lambda_e f_e $ is the constant zero function only if $ \lambda_e=0 $ for each $ e\in I $.
\begin{theorem}[{Afzali and Bowler \cite[Corollary 3.4]{borujeni2015thin}}]\label{thm: thin cofin}
The notion of thin independence in a thin family of functions gives rise to a cofinitary matroid.
\end{theorem}
For more information about infinite matroids, we refer to \cite{nathanhabil}.
\section{Proof of the main result}
We show by a relatively simple dualization argument that Theorem \ref{thm: AP Base exch} remains true for cofinitary 
matroids. 
\begin{theorem}\label{thm: dual of AP}
If $ M $ is a cofinitary matroid, then for every bases $ B_0 $ and $ B_1 $ of $ M $, there is a bijection $ f: B_0 
\rightarrow B_1 $ 
such that $ B_0\setminus \{ x \}\cup \{ f(x) \} $ is a base of $ M $ for each $ x\in B_0 $.
\end{theorem}
\begin{proof}
First we show that we can assume without loss of generality that $ B_0 $ and $ B_1 $ are disjoint and $ B_0\cup B_1=E(M) 
$. 
Indeed, suppose we already know this special case of the theorem and let $ M' $ be the matroid that we obtain from $ M $ by 
contracting $ B_0\cap B_1 $ and deleting 
 $E\setminus (B_0\cup B_1) $.  Then $ M' $ is still cofinitary, moreover,  $ B_0\setminus B_1 $ and $ B_1\setminus B_0 $ are 
 disjoint bases of $ M' $ whose union is $ E(M') $. Let $ f' $ be a function  that we obtain by applying the assumed special case of the theorem with 
 $ M' 
 $, $ B_0\setminus B_1 $ and $ 
 B_1\setminus B_0 $. Then the extension $ f $ of $ f' $ to $ B_0 $ where $ f(x):=x $ for every $ x\in B_0 \cap B_1 $ is as 
 desired.
 
Note that $ M^{*} $ is a finitary matroid and under our assumption $ B_0 $ and $ B_1 $ are bases of $ M^{*} $. Let $ g: 
B_1 \rightarrow  B_0 $  be a bijection that we obtain by applying  Theorem \ref{thm: AP Base exch} with $ M^{*} $, $ B_1 
$ 
and $ B_0 $ (i.e. the roles of the bases are switched). Then $ B_1 \setminus \{ y \} \cup \{ g(y) \} $ is a base of $ M^{*} $ for each $ y\in B_1 $. 
This 
means that its complement $ B_0\setminus \{ g(y) \}\cup \{ y \} $ is a base of $ M $ for each $ y\in B_1 $. By substituting $ y $ with $ g^{-1}(x) 
$, we conclude that $ 
B_0\setminus \{ x \}\cup \{ g^{-1}(x) \} $ is a base of $ M $ for each $ x\in B_0 $.  Thus
$ f:=g^{-1} $ 
is suitable.
\end{proof}

\begin{observation}\label{obs: reform}
For a matroid $ M $ the following are equivalent: 
\begin{enumerate}[label=(\roman*)]
\item For every bases $ B_0 $ and $ B_1 $ of $ M $, there is a bijection $ f: B_0 
\rightarrow B_1 $ 
such that $ B_0\setminus \{ x \}\cup \{ f(x) \} $ is a base of $ M $ for each $ x\in B_0 $.
\item For every bases $ B_0 $ and $ B_1 $ of $ M $, there is a bijection $ f: B_0 
\rightarrow B_1 $ 
such that $ B_1\setminus \{ f(x) \}\cup \{ x \} $ is a base of $ M $ for each $ x\in B_0 $.
\end{enumerate}

\end{observation}
\begin{proof}
To derive one property from the other, apply the assumed property while switching the roles of $ B_0 $ and $ B_1 $, then 
take the inverse of the resulting function.
\end{proof}
\begin{corollary}\label{cor: reduced dual of AP}
If $ M $ is a cofinitary matroid, $ J $ is independent in $ M $ and  $ B $ is a base of $ M $, then there is an injection $ f: J
\rightarrow B $ such that 
$ B \setminus \{ f(x) \}\cup \{ x \} $ is a base for each $ x\in B $.
\end{corollary}
\begin{proof}
Extend $ J $ to a base and apply Theorem \ref{thm: dual of AP} combined with Observation \ref{obs: reform}.
\end{proof}

\begin{proof}[Proof of Theorem \ref{thm: main}]
For $ j\in J $, let $ f_j: I\rightarrow \mathbb{F} $ be defined as $ f_{j}(i):=a_{i,j} $. For $ i\in I $, let $ f_i: I\rightarrow \mathbb{F} $ be the 
function for which $ f_i(i)=1 $ and $ f_i(i')=0 $ for $ i'\neq i $.\footnote{We assume 
that the index sets $ I $ and $ J $ are disjoint.}  Then $ \{ f_k:\ k\in I\cup J \} $ is a thin family of functions, thus by Theorem \ref{thm: thin cofin} 
it defines a cofinitary matroid $ M $ on $ I \cup J $ via thin independence. Moreover, $ I $ is base of $ M $ and $ J $ is independent in it. Thus 
Corollary 
\ref{cor: reduced dual of AP} applied with 
 $ J $ and  base $ I $ 
gives an injection $ \varphi: J 
\rightarrow I $ such that 
$ I \setminus \{ \varphi(j) \}\cup \{ j \} $ is a base of $ M $ for each $ j\in J $. But then we must have $ f_{j}(\varphi(j))\neq 0$ since otherwise $ 
\varphi(j) $  is not 
spanned by  $ I\setminus \{ \varphi(j) \}\cup \{ j \} $ in $ M $ because $ f_{\varphi(j)}(\varphi(j))=1 $ while $ f_k(\varphi(j))=0 $ for every $ k\in 
(I\setminus \{ \varphi(j) \}\cup \{ j \}) $, 
contradicting that $ I\setminus \{ \varphi(j) \}\cup \{ j \} $ is a base of $ M $. By definition this means that $ a_{\varphi(j),j}\neq 0 $ for every $ 
j\in J $ as desired.
\end{proof} 
\printbibliography
\end{document}